\documentclass[11pt]{article}

\usepackage{latexsym,color,amsmath,amsthm,amssymb,amscd,amsfonts}
\usepackage{graphicx}
\usepackage{cancel}

\newtheorem{conj}{Conjecture}[section]
\newtheorem{theorem}[conj]{Theorem}

\newtheorem{remark}[conj]{Remark}
\newtheorem{lemma}[conj]{Lemma}
\newtheorem{prop}[conj]{Proposition}

\newcommand\independent{\protect\mathpalette{\protect\independent}{\perp}} 
\def\independent#1#2{\mathrel{\rlap{$#1#2$}\mkern2mu{#1#2}}}

\newcommand{\R}{\mathbb{R}}

\renewcommand{\P}{\mathbb{P}}

\newcommand{\E}{\mathbb{E}}

\DeclareMathOperator{\Var}{Var}

\newcommand{\vol}{\mathrm{Vol}}

\newcommand{\eps}{\varepsilon}

\date{\vspace{-5ex}}

\author{Arnaud Marsiglietti and Puja Pandey}

\title{On the Equivalence of Statistical Distances for Isotropic Convex Measures}

\begin{document}

\maketitle

\begin{abstract}
    
    We establish quantitative comparisons between classical distances for probability distributions belonging to the class of convex probability measures. Distances include total variation distance, Wasserstein distance, Kullback-Leibler distance and more general R\'enyi divergences. This extends a result of Meckes and Meckes (2014).
    
\end{abstract}


\section{Introduction}

In convex geometry and its probabilistic aspects, many fundamental inequalities are shown to be reversed up to universal constants in the presence of geometric properties, such as convexity. Examples include reverse H\"older and Jensen type inequalities (see, e.g., \cite{FP}, \cite{Fa}, \cite{Be}, \cite{KPB}, \cite{L}, \cite{Bor73}, \cite{H}), reverse isoperimetric inequalities (see, e.g., \cite{Ba}, \cite{Bar}, \cite{BH}), and reverse Brunn-Minkowski inequalities (see, e.g., \cite{M1}, \cite{M2}, \cite{P}, \cite{BM}). Long-standing conjectures, such as the Mahler conjecture \cite{Mal} and Bourgain's hyperplane conjecture \cite{Bou} also are related to the reversal of fundamental inequalities.

Another important example is the equivalence of distances between probability distributions established by Meckes and Meckes \cite{MM}, who showed that under a log-concavity assumption, many classical distances are comparable. The goal of this article is to extend their results to a broader class of probability measures, called convex measures.

The class of convex measures contains fundamental distributions in probability and statistics. Examples include Gaussian distributions, uniform distributions on a convex set and more general log-concave distributions, as well as heavy tailed distributions such as Cauchy type of the form
$$ f(x) = \frac{C}{(1 + |x|^2)^{\frac{n+\beta}{2}}}, \quad x \in \R^n, $$
where $\beta > 0$ is a parameter, $C > 0$ is the normalizing constant, and $|\cdot|$ denotes the Euclidean norm on $\R^n$, $n \geq 1$.

The classical distances between probability distributions we consider are the bounded Lipschitz distance, the total variation distance, the Wasserstain distance, the relative entropy and more general R\'enyi and Tsallis divergences. More precisely, given probability measures $\mu$ and $\nu$ on $\R^n$, the bounded Lipschitz distance between $\mu$ and $\nu$ is defined as
$$ d_{BL}(\mu,\nu) = \sup_{\|g\|_{BL} \leq 1} \left| \int g \, d\mu - \int g \, d\nu  \right|, $$
where for a function $g \colon \R^n \to \R$,
$$ \|g\|_{BL} = \max \left\{ \|g\|_{\infty}, \, \sup_{x \neq y} \frac{|g(x)-g(y)|}{|x-y|} \right\}. $$
The total variation distance between $\mu$ and $\nu$ is defined as
$$ d_{TV}(\mu, \nu) = 2 \sup_{A \subset \R^n} |\mu(A) - \nu(A)|. $$
The $p$-th Wasserstein distance, $p \geq 1$, between $\mu$ and $\nu$ is defined as
$$ W_p(\mu,\nu) = \inf_{(X,Y) \, : \, X \sim \mu, Y \sim \nu} \E[|X-Y|^p]^{\frac{1}{p}}, $$
where the infimum is taken over all joint random variables $(X,Y)$ with marginal $X$ (resp. $Y$) distributed according to $\mu$ (resp. $\nu$). The R\'enyi divergence of order $p > 0$ between a measure $\mu$ with density $f$ (with respect to Lebesgue measure on $\R^n$) and $\nu$ with density $g$ is defined as
$$ D_{p}(\mu || \nu) = \frac{1}{p - 1} \log \left( \int \left(\frac{f(x)}{g(x)}\right)^{p} g(x) dx \right). $$
This family of distances includes the relative entropy (or Kullback-Leibler distance)
$$ D_{1}(\mu || \nu) = D(\mu || \nu) = \int f(x) \log \left( \frac{f(x)}{g(x)} \right) dx, $$
and is related to the family of Tsallis entropies
$$ T_{p}(\mu || \nu) = \frac{1}{p - 1} (e^{(p - 1) D_{p}(\mu || \nu)} - 1). $$
There are known relationships between these distances. For example,
$$ d_{BL}(\mu, \nu) \leq \min\{ d_{TV}(\mu,\nu), W_1(\mu,\nu)\}, $$
which follows from a dual representation of the total variation distance and Wasserstein distance (see, e.g., \cite{V}, \cite{MM}), and 
$$ W_p(\mu,\nu) \leq W_q(\mu,\nu) $$
for all $p \leq q$, by H\"older's inequality. As for the entropic quantities, Gilardoni \cite{Gil} proved that for all $p \in (0,1]$,
$$ \frac{p}{2} d_{TV}(\mu,\nu)^2 \leq D_{p}(\mu || \nu), $$
which extends a result of Pinsker \cite{Pin} and Csisz\'ar \cite{C}. When $\nu = \gamma_n$, the standard Gaussian measure in $\R^n$, Talagrand \cite{Tal} proved that
$$ W_2^2(\mu,\gamma_n) \leq 2 D(\mu || \gamma_n). $$
It turns out that R\'enyi divergences are comparable in the range $(0,1)$. For all $0 < p < q < 1$,
$$ \frac{p (1-q)}{(1-p)^2} D_{q}(\mu || \nu) \leq D_{p}(\mu || \nu) \leq D_{q}(\mu || \nu), $$
see, e.g., \cite{EH}, \cite{BCG}. The case $p > 1$ is more intricate. For example, consider a one-dimensional exponential distribution $\mu$, then, for any $p > 1$,
$$ D_{p}(\mu,\gamma_1) = +\infty, $$
while $D(\mu,\gamma_1) < +\infty$. Hence, even among log-concave distributions, there may not be an absolute comparison between R\'enyi entropies of order $p \geq 1$. Additional assumptions are thus necessary. Nevertheless, for all $0 < p \leq q$,
$$ D_{p}(\mu || \nu) \leq D_{q}(\mu || \nu), $$
and similarly for $T_{p}$ (see \cite{BCG}). Moreover, one clearly has
$$ D_{p}(\mu || \nu) \leq T_{p}(\mu || \nu). $$

Classical counterexamples show that, in general, the above inequalities cannot be reversed, and that there are no comparison between the total variation distance and the Wasserstein distance. The goal of this article is to show that all of the above distances are equivalent when restricted to the class of convex measures. In Section 2, we recall the definition and the main properties of convex measures. Section 3, which contains our main results, establishes a quantitative comparison between all aforementioned distances within the class of isotropic convex measures.

\section{Preliminaries on convex measures}

\subsection{Definition}

For a parameter $\alpha \in [-\infty,+\infty]$, for real numbers $a,b \geq 0$, and $\lambda \in [0,1]$, denote
\[ M_{\alpha}^{\lambda}(a,b) = \left\{
\begin{array}{llll} \left( (1-\lambda)a^{\alpha} + \lambda b^{\alpha} \right)^{\frac{1}{\alpha}} & \mbox{ if $\alpha \notin \{-\infty,0,+\infty\}$} \\
\min(a,b) & \mbox{ if $\alpha=-\infty$} \\ a^{1-\lambda} b^{\lambda} & \mbox{ if $\alpha=0$} \\ \max(a,b) & \mbox{ if $\alpha=+\infty$}
\end{array}.
\right. \]
Recall that a random variable $X$ in $\R^n$ with distribution $\mu$ is {$s$-concave}, $s \in [-\infty,+\infty]$, if for all $\lambda \in [0,1]$, for all compact sets $A,B \subset \R^n$ such that $\mu(A)\mu(B)>0$, one has
\begin{eqnarray}\label{s-concave}
\mu((1-\lambda)A + \lambda B) \geq  M_s^{\lambda}(\mu(A),\mu(B)).
\end{eqnarray}
The parameter $s$ is understood as a convexity parameter. From the definition, one can see by Jensen's inequality that any $s$-concave measure is $r$-concave for all $r \leq s$. In particular, any $s$-concave measure is $-\infty$-concave. The class of $-\infty$-concave measures is called convex measures, and the class of $0$-concave measures is called log-concave measures. A function $f \colon \R^n \to [0,+\infty)$ is {$\kappa$-concave}, $\kappa \in [-\infty,+\infty]$, if for all $\lambda \in [0,1]$, for all $x,y \in \R^n$ such that $f(x)f(y)>0$, one has
\begin{eqnarray}
f((1-\lambda)x + \lambda y) \geq M_{\kappa}^{\lambda}(f(x),f(y)).
\end{eqnarray}
The class of convex measures has been extensively studied by Borell in \cite{Bor74}, \cite{Bor75}. In particular, Borell proved that if $X$ is not supported on a proper affine subspace of $\R^n$, then $X$ is $s$-concave, with $s \in [-\infty, \frac{1}{n}]$, if and only if $X$ admits a density $f$ with respect to Lebesgue measure on $\R^n$, which is $\kappa$-concave, with $\kappa \in [-\frac{1}{n}, +\infty]$ satisfying the relation $\kappa = \frac{s}{1-sn}$. Moreover, if $X$ is $s$-concave, then the random variable $\langle X, \theta \rangle$ is also $s$-concave, for all $\theta \in S^{n-1}$.

\subsection{Concentration inequalities}

Recall that a random variable $X$ in $\R^n$ is isotropic if $X$ is centered and if for all $\theta \in S^{n-1}$,
$$ \E[\langle X, \theta \rangle^2] = 1. $$

The next two lemmas provide concentration and moments inequalities for $s$-concave measures, and were established in \cite{AGLL}. First, recall that if $X$ is $s$-concave, then $\E[|X|^p]< +\infty$ for all $s > -\frac{1}{p}$ (see \cite{Bor74}), where $|\cdot|$ denotes the Euclidean norm in $\R^n$.

\begin{lemma}\textnormal{(\cite[Corollary 5.4]{AGLL})}\label{deviation}
Let $-\frac{1}{2} < s < 0$. Let $X$ be an isotropic $s$-concave random variable in $\R^n$. Then, for all $u>0$,
$$ \P(|X| \geq u) \leq \left( \frac{c \max\{\sqrt{n}, \frac{1}{|s|}\}}{u} \right)^{\frac{1}{2|s|}}. $$
In particular, if $s \geq -\frac{1}{2\sqrt{n}}$, then for every $6c \sqrt{n} \leq u \leq \frac{3c}{|s|}$,
$$ \P(|X| \geq u) \leq e^{-c_0 u}, $$
where $c$ and $c_0$ are universal positive constants.

\end{lemma}

\begin{lemma}\textnormal{(\cite[Lemma 7.3]{AGLL})}\label{norm}
Let $p \geq 1$. Let $-\frac{1}{p} < s < 0$. Let $X$ be an $s$-concave random variable in $\R^n$. Then, there is a universal constant $c>0$ such that
$$ E[|X|^{p}]^{\frac{1}{p}} \leq c \, C(p, s) \E[|X|], $$
where
\begin{equation}\label{norm-cste}
C(p,s)=\left\{
\begin{array}{cl}
 p & \mbox{ for } \, s > - \frac{1}{p+1} \\
\frac{1}{|s|^{1-\frac{1}{p}}(1-p |s|)^{\frac{1}{p}}} & \mbox{ for } \, -\frac{1}{p} < s \leq - \frac{1}{p+1} \\
\end{array}. \right.
\end{equation}

\end{lemma}

The next lemma provides bounds on the var-entropy of $s$-concave random variables, and was established in \cite{FLM}.

\begin{lemma}[\cite{FLM}]\label{variance}

Let $\kappa \in (-\frac{1}{n}, +\infty]$. Let $X$ be a random variable in $\R^{n}$ with density $f$ being $\kappa$-concave. Then,
$$ \Var( \log(f(X))) \leq \sum_{i=1}^{n} \frac{1}{(1 + i \kappa)^2}. $$

\end{lemma}

The next lemma is implicit in \cite{FG06} and \cite{F-hdr}. It is an extension of a result of Gr\"unbaum \cite{Gr}. We include a proof for reader convenience.

\begin{lemma}\label{lem-dev}

Let $s > -1$. Let $X$ be an $s$-concave random variable in $\R$. Then,
$$ \P(X \geq \E[X]) \geq (1+s)^{-\frac{1}{s}}. $$

\end{lemma}

\begin{proof}
Assume $s<0$, the argument for $s \geq 0$ is similar. Since $X$ is $s$-concave, $-X$ is $s$-concave and therefore the cumulative distribution function of $-X$, $F(x) = \P(-X \leq x)$, $x \in \R$, is $s$-concave. Therefore, $F^s$ is convex. Denote by $f$ the density of $-X$. By Jensen's inequality,
$$ F^s(\E[-X]) = F^s \left( \int_{\R} t f(t) dt \right) \leq \int_{\R} F^s(t) f(t) dt = \left.\frac{F^{s+1}(x)}{s+1}\right|_{x=-\infty}^{x=+\infty} = \frac{1}{1+s}. $$
Equivalently, $\P(-X \leq \E[-X]) \geq (1+s)^{-\frac{1}{s}}$.
\end{proof}

\subsection{Maximum of the density of convex measures}

It is known that the density of a convex measure is bounded (see, e.g., \cite{Bob}). This section gathers and develops explicit bounds on the maximum of the density of isotropic $s$-concave distributions. In dimension 1, there is the following bound.

\begin{lemma}[\cite{BZ}]\label{max-density}

Let $s \in (-\frac{1}{2}, 0)$. Let $X$ be an isotropic $s$-concave random variable in $\R$ with density $f$. Then,
$$ \|f\|_{\infty} \leq \frac{1}{1 + 2s}. $$

\end{lemma}

Next, we develop a multidimensional analog of Lemma \ref{max-density}.

\begin{prop}\label{max-dens}

Let $s \in (-\frac{1}{2},0)$. Let $X$ be an isotropic $s$-concave random variable in $\R^n$ with density $f$. Then,
$$ \|f\|_{\infty} \leq c^{n(1+n|s|)} d_0^n n^{\frac{n}{2}}, $$
where $c>0$ is a universal constant and
\begin{equation}\label{c_0-}
    d_0 = d_0(n,s) = \frac{(1+n|s|)^{4(1+n|s|)}}{1+2s}.
\end{equation}  

\end{prop}

Proposition \ref{max-dens} extends \cite[Theorem 9(e)]{BZ} to the whole range $s \in (-\frac{1}{2}, 0)$ and provides a simpler estimate. Note that the constant $d_0$ in Proposition \ref{max-dens} becomes absolute when $s >-\frac{c}{n}$. The proof relies on the following lemma.

\begin{lemma}[\cite{FLM}]\label{large-dev}

Let $s \in (-\infty,0)$. Let $X$ be a random variable in $\R^n$ with density $f$ being $\kappa$-concave, with $\kappa=s/(1-sn)$. For any $c_0 \in (0,1)$ such that $n \log(c_0) < -\sum_{i=1}^n (1+i \kappa)^{-1}$, there exists $c_1 \in (0,1)$ such that
$$ \P(f(X) \geq c_0^n \|f\|_{\infty}) \geq 1-c_1^n. $$

\end{lemma}

\begin{proof}[Proof of Proposition \ref{max-dens}]
The proof of Lemma \ref{large-dev} in \cite{FLM} provides information on the constant $c_1$. Precisely, one may take
\begin{equation}\label{c_1}
     c_1 = c_0^{\alpha} \prod_{i=1}^n \left( \frac{1+i \kappa}{1+i \kappa - \alpha} \right)^{\frac{1}{n}},
\end{equation}
where $\alpha \in (0, 1 + n \kappa)$ satisfies
$$ \sum_{i=1}^n \frac{1}{1+i \kappa - \alpha} = -n\log(c_0). $$
Note that by the AM-GM inequality, and the simple inequality $\log(x) < \frac{2}{\alpha} x^{\alpha/2}$, for $x>0$, we have
\begin{equation}\label{estim}
     c_1 \leq c_0^{\alpha} \frac{1}{n} \sum_{i=1}^n \frac{1+i \kappa}{1+i \kappa - \alpha} \leq c_0^{\alpha} \frac{1}{n} \sum_{i=1}^n \frac{1}{1+i \kappa - \alpha} = c_0^{\alpha} \log \left( \frac{1}{c_0} \right) < \frac{2}{\alpha} c_0^{\frac{\alpha}{2}}.
\end{equation}
Moreover, if $\alpha < \frac{1 + n\kappa}{2}$, then
$$ -\log(c_0) = \frac{1}{n} \sum_{i=1}^n \frac{1}{1+i \kappa - \alpha} \leq \frac{1}{1+n \kappa - \alpha} < \frac{2}{1 + n \kappa}. $$
We deduce that if $-\log(c_0) \geq 2/(1+ n \kappa)$, then $\alpha \geq (1+n\kappa)/2$. Now, choose
\begin{equation}\label{c_0}
    c_0 = \left( \frac{1+n\kappa}{4} \frac{(1+s)^{-\frac{1}{s}}}{2} \right)^{\frac{4}{1+n\kappa}},
\end{equation}
and recall that $\kappa = \frac{s}{1-sn} \in (-\frac{1}{n+2}, 0)$ and $s \in (-\frac{1}{2}, 0)$. Note that this choice of $c_0$ satisfies $0 < c_0 < 8^{-\frac{4}{1+n \kappa}}$, therefore $c_0 \in (0,1)$, $-\log(c_0) \geq 2/(1+ n \kappa) \geq \frac{1}{n} \sum_{i=1}^n (1+i \kappa)^{-1}$, and thus the corresponding $\alpha$ such that $\sum_{i=1}^n \frac{1}{1+i \kappa - \alpha} = -n\log(c_0)$ satisfies $\alpha \geq (1+n \kappa)/2$.
Hence, recalling \eqref{c_1}, we have by \eqref{estim},
\begin{equation}\label{star}
    c_1 < \frac{2}{\alpha} c_0^{\frac{\alpha}{2}} \leq \frac{4}{1+n \kappa} c_0^{\frac{1+n\kappa}{4}} = \frac{(1+s)^{-\frac{1}{s}}}{2}.
\end{equation}
Now, consider the convex set
$$ K = \{x \in \R^n : f(x) \geq c_0^n \|f\|_{\infty}\}, $$
where $c_0$ is given in \eqref{c_0}, and choose
\begin{equation}\label{c_2}
    c_2 = (1 + 2s) \frac{(1+s)^{-\frac{1}{s}}}{2}.
\end{equation}
We will prove that $K \supset c_2 B_2^n$. For this, we follow \cite{K}. Assume that $K$ does not contain $c_2 B_2^n$. Since $K$ is convex, this implies that there exists $\theta \in S^{n-1}$ such that $K \subset \{x \in \R^n : \langle x,\theta \rangle \leq c_2 \}$. Therefore, by Lemma \ref{large-dev} and \eqref{star},
$$ \P( \langle X,\theta \rangle \leq c_2) \geq \P(X \in K) \geq 1 - c_1^n \geq 1-c_1 > 1 - \frac{(1+s)^{-\frac{1}{s}}}{2}. $$
However, denoting by $g$ the density of $\langle X,\theta \rangle$ and recalling \eqref{c_2}, we have by Lemmas \ref{lem-dev} and \ref{max-density},
\begin{eqnarray*}
\P( \langle X,\theta \rangle \leq c_2) & = & \P(\langle X,\theta \rangle \leq 0) + \P(0 \leq \langle X,\theta \rangle \leq c_2) \\ & \leq & 1-(1+s)^{-\frac{1}{s}} + \|g\|_{\infty} c_2 \\ & \leq & 1-(1+s)^{-\frac{1}{s}} + \frac{1}{1+2s} c_2 \\ & = & 1 - \frac{(1+s)^{-\frac{1}{s}}}{2}.
\end{eqnarray*}
Hence, we have a contradiction. Therefore, $K \supset c_2 B_2^n$. We deduce that
$$ 1 \geq \int_K f(x) dx \geq c_0^n \|f\|_{\infty} \vol(K) \geq c_0^n \|f\|_{\infty} c_2^n \vol(B_2^n). $$
It remains to note that one may find a universal constant $c>0$ such that
\begin{equation}\label{c0-c2}
    \frac{1}{c_0} \leq [c(1+n|s|)]^{4(1+n|s|)}, \qquad \frac{1}{c_2} \leq \frac{c}{1+2s}.
\end{equation}
\end{proof}

\subsection{$L^1$-regularization of the density of convex measures}

The following proposition extends a result of Eldan and Klartag \cite{EK} to convex measures. First, recall the density of a centered Gaussian in $\R^n$ with variance $t^2$, $t>0$,
\begin{equation}\label{gauss-t}
    \phi_t(x) = \frac{1}{(2\pi t^2)^{\frac{n}{2}}} e^{-\frac{|x|^2}{2t^2}}, \quad x \in \R^n.
\end{equation}

\begin{prop}\label{regul}

Let $s \in (-\frac{1}{2},0)$ and let $f$ be the density of an isotropic $s$-concave measure in $\R^n$. Recall the value of $d_0$ in \eqref{c_0-}. Then, there is a universal constant $c>0$ such that for all $t>0$,
\begin{equation}\label{convol}
\|f - f * \phi_{t}\|_{L^1} \leq c^{1+n|s|} d_0 t n.
\end{equation}

\end{prop}

\begin{proof}
First, let us show that one may assume that $f$ is of class $C^1$ and strictly positive on $\R^n$. Since $f$ is $\kappa$-concave, $\kappa = \frac{s}{1-sn} < 0$, we have that the function $F = f^{\kappa}$ is convex. Define, for $\eps > 0$, the Moreau envelope of F (also called infimum convolution),
$$ F_{\eps}(x) = \inf_{y \in \R^n} \{F(y) + \frac{1}{2 \eps} |x-y|^2 \}, \quad x \in \R^n. $$
It is known that for all $\eps > 0$, $F_{\eps}$ is convex, of class $C^1$, and finite on $\R^n$, and $F_{\eps} \nearrow F$ pointwise as $\eps \searrow 0$ (see, e.g., \cite{Mo}, \cite{Br}, \cite{JTZ}, \cite{Bec}). Therefore, defining $f_{\eps} = F_{\eps}^{\frac{1}{\kappa}}$ gives rise of a family of $\kappa$-concave functions of class $C^1$, strictly positive on $\R^n$, and converging pointwise to $f$. Since $f_{\eps} \leq f_1$ for all $\eps \in (0,1)$, if one can show that $f_1 \in L^1(dx)$, then one may apply Lebesgue dominated convergence theorem to deduce that $f_{\eps}$ converges to $f$ in $L^1(dx)$ as $\eps \to 0$. This would conclude the argument that one may restrict the proof to $\kappa$-concave density functions that are $C^1$ and strictly positive on $\R^n$. To show that $f_1 \in L^1(dx)$, we use that since $f$ is the density of an $s$-concave measure,
$$ f(x) \leq \frac{C}{1+|x|^{n-\frac{1}{s}}}, \quad x \in \R^n, $$
for some constant $C>0$ (see, e.g., \cite{Bob}). Therefore,
\begin{eqnarray*}
F_1(x) = \inf_{y \in \R^n} \{f^{\kappa}(y) + \frac{1}{2} |x-y|^2 \} & \geq & \inf_{y \in \R^n} \{ C^{\kappa} (1+|y|^{n-\frac{1}{s}})^{|\kappa|} + \frac{1}{2} |x-y|^2 \} \\ & \geq & \inf_{y \in \R^n} \{ \frac{C^{\kappa}}{2^{1-|\kappa|}} (1+|y|) + \frac{1}{2} |x-y|^2 \} \\ & = & \frac{C^{\kappa}}{2^{1-|\kappa|}}( 1 + \inf_{y \in \R^n} \{ |y| + \frac{1}{2\frac{C^{\kappa}}{2^{1-|\kappa|}}} |x-y|^2 \},
\end{eqnarray*}
where the last inequality comes from concavity of $x \mapsto x^{|\kappa|}$, as $|\kappa| = \frac{|s|}{1-sn} < \frac{1}{n+2} < 1$, and $(n-\frac{1}{s})|\kappa| = 1$. We recognize the Moreau envelope of the Euclidean norm, which is known to be the Huber function (see, e.g., \cite[Chapter 6]{Bec}), namely
$$ \inf_{y \in \R^n} \{ |y| + \frac{1}{2\frac{C^{\kappa}}{2^{1-|\kappa|}}} |x-y|^2 \} = H_{\lambda}(|x|), $$
with $\lambda = \frac{C^{\kappa}}{2^{1-|\kappa|}}$, where, for $r \geq 0$,
\[ 
H_{\lambda}(r) = \left\{
  \begin{array}{lr} 
      r - \frac{\lambda}{2} & r > \lambda \\
      \frac{1}{2\lambda}r^2 & r \leq \lambda 
      \end{array}.
\right.
\]
Finally,
$$ f_1(x) = F_1^{\frac{1}{\kappa}}(x) \leq \frac{2^{\frac{1-|\kappa|}{|\kappa|}} C}{(1+H_{\lambda}(|x|))^{\frac{1}{|\kappa|}}}, $$
which is an integrable function since $\frac{1}{|\kappa|} > n$.

Now, let us prove inequality \eqref{convol} for $C^1$ and strictly positive $\kappa$-concave functions, $\kappa = \frac{s}{1-sn}$. For this, we follow \cite{EK}. Recall from the proof of Proposition \ref{max-dens} that the set
$$ K = \{x \in \R^n : f(x) \geq c_0^n \|f\|_{\infty}\} $$
contains $c_2 B_2^n$, where $c_0$ is defined in \eqref{c_0} and $c_2$ in \eqref{c_2}. Denoting $F=f^{\kappa}$, we have that $F$ is convex and thus for all $x,y \in \R^n$,
$$ \langle \nabla F(x) , y\rangle \leq \langle \nabla F(x) , x \rangle + F(y) - F(x). $$
Taking $y = c_2 \nabla F(x)/|\nabla F(x)|$ when $|\nabla F(x)| \neq 0$, we deduce that for all $x \in \R^n$,
$$ c_2 |\nabla F(x)| \leq \langle \nabla F(x) , x \rangle + \sup_{|y| \leq c_2} F(y) - \inf_{x \in \R^n} F(x) \leq \langle \nabla F(x) , x \rangle + \sup_{y \in K} F(y) - \inf_{x \in \R^n} F(x). $$
Therefore,
$$ c_2 |\kappa| |\nabla f(x)| \leq \kappa \langle \nabla f(x) , x \rangle +  f^{1-\kappa}(x) \|f\|^{\kappa}_{\infty} (c_0^{n \kappa} - 1). $$
Hence, for all $x \in \R^n$,
$$ |\nabla f(x)| \leq - \frac{1}{c_2} \langle \nabla f(x) , x \rangle + f(x) \frac{c_0^{n \kappa} - 1}{c_2 |\kappa|}. $$
Integrating the above inequality and using an integration by parts, we obtain
$$ \int_{\R^n} |\nabla f(x)| dx \leq \frac{n}{c_2} \left( 1 + \frac{c_0^{n \kappa} - 1}{n |\kappa|} \right) \leq \frac{n}{c_0 c_2}. $$
We conclude by using the estimate \eqref{c0-c2} and the following particular case of a result of Ledoux \cite{Le} (valid for all $C^1$ functions),
\begin{equation*}
    \|f - f * \phi_{t}\|_{L^1} \leq 2 t \int_{\R^n} |\nabla f(x)| dx.
\end{equation*}
\end{proof}

\section{Main results and proofs}

First, we introduce the next elementary lemma, which will be implicitly used.

\begin{lemma}\label{minimize}

Let $A,B, m, p, M > 0$. Define $F(t) = At^m + \frac{B}{t^p}$. Then,
$$ \inf_{t \geq M} F(t) = A^{\frac{p}{m+p}} B^{\frac{m}{m+p}} \left( \max \left\{\frac{A}{B} M^{m+p}, \frac{p}{m} \right\}^{\frac{m}{m+p}}  + \frac{1}{\max \left\{\frac{A}{B} M^{m+p}, \frac{p}{m} \right\}^{\frac{p}{m+p}} }  \right). $$

\end{lemma}

\begin{proof}
The infimum is attained at $t = \max \left\{M, \left( \frac{B}{A} \frac{p}{m} \right)^{\frac{1}{m+p}} \right\}$.
\end{proof}

The first theorem provides quantitative reversal bounds between total variation distance and bounded Lipschitz distance.

\begin{theorem}\label{tv-bl}

Let $s \in (- \frac{1}{2},0)$. Let $\mu$ and $\nu$ be $s$-concave isotropic probability measures on $\R^{n}$. Then, there exists a universal constant $c>0$ such that
$$ d_{TV}(\mu,\nu) \leq c^{1+n|s|} \frac{(1+n|s|)^{2(1+n|s|)}}{\sqrt{1+2s}} \sqrt{n} \sqrt{d_{BL}(\mu,\nu)}. $$

\end{theorem}

\begin{proof}
Similarly as in \cite{MM}, let $g$ be a continuous function with $\|g\|_{\infty} \leq 1$. For $t > 0$, let $g_{t} = g * \phi_{t}$, where $\phi_{t}$ is defined in \eqref{gauss-t}. Note that $||g_{t}||_{\infty} \leq 1$ and that $g_{t}$ is $1/t$-Lipschitz. By triangle inequality,
\begin{align*}
    \left| \int g d\mu - \int g d\nu \right| \leq \left| \int \left(g - g_{t}\right) d\mu \right| + \left| \int \left(g - g_{t}\right) d\nu \right| + \left| \int g_{t} d\mu - \int g_{t} d\nu \right|.
\end{align*}
Denote by $f$ the density of $\mu$. By Proposition \ref{regul}, we have
\begin{align*}
    \left| \int(g - g_{t}) d\mu \right| = \left| \int g(f - f * \phi_{t}) dx \right| \leq ||f - f * \phi_{t}||_{1} \leq \frac{2 t n}{c_0 c_2}.
\end{align*}
Similarly,
$$ \left| \int(g - g_{t}) d\nu \right| \leq \frac{2 t n}{c_0 c_2}. $$
Finally,
\begin{align*}
    \left| \int g_{t} d\mu - \int g_{t} d\nu \right| \leq d_{BL}(\mu,\nu) ||g_{t}||_{BL} \leq d_{BL}(\mu,\nu) \max(1,1/t).
\end{align*}
Combining the above estimates, taking the supremum over all such $g$, and using the dual representation
$$ d_{TV}(\mu, \nu) = \sup \left\{ \left|\int g d \mu - \int g d\nu \right| : g \in C(\R^n), \, \|g\|_{\infty} \leq 1 \right\} $$
of total variation distance (see, e.g., \cite{MM}), we deduce
\begin{align*}
    d_{TV}(\mu,\nu) \leq d_{BL}(\mu,\nu) \max(1,1/t) + \frac{4 n}{c_0 c_2}t
\end{align*}
for every $t>0$. To conclude, choose
$$ t = \sqrt{\frac{d_{BL}(\mu,\nu)}{\frac{4 n}{c_0 c_2}}}, $$
and note that $t \leq 1$ since $c_0, c_2 \in (0,1)$ and $d_{BL}(\mu,\nu) \leq 2$, then apply \eqref{c0-c2}.
\end{proof}

\begin{remark}

We recover the result of \cite{MM} for log-concave measures by letting $s \to 0$ in Theorem \ref{tv-bl}.

\end{remark}

Next, we provide a comparison between 1-Wasserstein distance and bounded Lipschitz distance.

\begin{theorem}\label{compar-BL}

Let $s \in (-\frac{1}{2},0)$. Let $\mu$ and $\nu$ be isotropic $s$-concave probability measures on $\R^{n}$. Then, there exists a universal constant $c>0$ such that
$$ W_1(X,Y) \leq c \sqrt{n} \max\{1, \frac{1}{\sqrt{n} |s|}\}^{\frac{1}{1+4|s|}} d_{BL}(X,Y)^{\frac{1}{1+4|s|}}. $$

\end{theorem}

\begin{proof}
We follow \cite{MM}. First, recall the representation
\begin{equation}\label{wasser-1}
    W_1(\mu, \nu) = \sup_g \left| \int g \, d\mu - \int g \, d\nu \right|, 
\end{equation}
where the supremum is over 1-Lipschitz functions $g \colon \R^n \to \R$ (see, e.g., \cite{V}). Let $g \colon \R^n \to \R$ be a 1-Lipshitz function. Assume without loss of generality that $g(0) = 0$. For $R>0$, define
\[ g_R(x) = \left\{
\begin{array}{cll}
 -R & \mbox{ if } \, g(x) < -R \\
g(x) & \mbox{ if } \, -R \leq g(x) \leq R \\
R & \mbox{ if } \, g(x) > R
\end{array}. \right.\]
By construction, we have
\begin{equation}\label{BL-norm}
    \|g_R\|_{BL} \leq \max\{1,R\}.
\end{equation}
Note that
$$ \E[|g(X) - g_R(X)|] = \E[(g(X)-R)1_{\{g(X)>R\}}] - \E[(g(X)+R)1_{\{g(X)<-R\}}] \leq  \E[|g(X)|1_{\{|g(X)|>R\}}]. $$
Since $|g(X)| \leq |X|$, we deduce by Cauchy-Schwarz that
$$ \E[|g(X) - g_R(X)|] \leq \E[|X|1_{\{|X|>R\}}] \leq \sqrt{\E[|X|^2]} \sqrt{\P(|X|>R)}. $$
Using that $X$ is isotropic and applying Lemma \ref{deviation}, there exists a universal constant $c>0$ such that
$$ \E[|g(X) - g_R(X)|] \leq \sqrt{n} \left( \frac{c \max\{\sqrt{n}, \frac{1}{|s|}\}}{R} \right)^{\frac{1}{4|s|}}. $$
The same inequality holds for $\E[|g(Y) - g_R(Y)|]$. We deduce that
\begin{eqnarray*}
|\E[g(X)] - \E[g(Y)]| \leq |\E[g_R(X)] - \E[g_R(Y)]| + 2 \sqrt{n} \left( \frac{c \max\{\sqrt{n}, \frac{1}{|s|}\}}{R} \right)^{\frac{1}{4|s|}}.
\end{eqnarray*}
Using the fact that $|\E[g_R(X)] - \E[g_R(Y)]| \leq \|g_R\|_{BL} d_{BL}(X,Y)$, we arrive at 
\begin{equation}\label{BL-1}
|\E[g(X)] - \E[g(Y)]| \leq \|g_R\|_{BL} d_{BL}(X,Y) + 2 \sqrt{n} \left( \frac{c \max\{\sqrt{n}, \frac{1}{|s|}\}}{R} \right)^{\frac{1}{4|s|}}.
\end{equation}
Taking supremum over all 1-Lipschitz function $g$, using \eqref{BL-norm} and the representation \eqref{wasser-1}, inequality \eqref{BL-1} leads to
\begin{equation}\label{BL-2}
W_1(X,Y) \leq \max\{1,R\} d_{BL}(X,Y) + 2 \sqrt{n} \left( \frac{c \max\{\sqrt{n}, \frac{1}{|s|}\}}{R} \right)^{\frac{1}{4|s|}}.
\end{equation}
Note that there exists a universal constant $c>0$ such that 
$$ d_{BL}(X,Y) \leq 2 \leq \frac{1}{4|s|} 2\sqrt{n} \left(c\max\{\sqrt{n}, \frac{1}{|s|}\} \right)^{\frac{1}{4|s|}}, $$
therefore, taking supremum over all $R \geq 1$ in \eqref{BL-2}, we deduce from Lemma \ref{minimize} that there is a universal constant $c > 0$ such that
$$
W_1(X,Y) \leq d_{BL}(X,Y)^{\frac{1}{1 + 4|s|}} \left( 2 \sqrt{n} \left( c \max\{\sqrt{n}, \frac{1}{|s|}\} \right)^{\frac{1}{4|s|}} \right)^{\frac{4|s|}{1 + 4|s|}} \left[ \left( \frac{1}{4|s|} \right)^{\frac{4|s|}{1+4|s|}} + \left( 4|s| \right)^{\frac{1}{1 + 4|s|}} \right]. $$
The result follows since $m^{-\frac{m}{1+m}} + m^{\frac{1}{1+m}} \leq 2$, with $m = 4|s| > 0$.
\end{proof}

\begin{remark}\label{rem-1}

As $s \to 0$, the constant in the right-hand side of Theorem \ref{compar-BL} blows up to $+\infty$. One may recover the result of \cite{MM} for log-concave measures as $s \to 0$ by applying the second part of Lemma \ref{deviation}. The details are left to the readers.

\end{remark}

The following theorem establishes a comparison between Wasserstein distances.

\begin{theorem}\label{compar-wass}

Let $1 \leq p < q$ and $\alpha \in (1,2]$. Let $s \in (-\frac{1}{\alpha q}, 0)$. Let $\mu$ and $\nu$ be isotropic $s$-concave probability measures on $\R^{n}$. Then, there is an absolute constant $c>0$ such that
$$ W_q(\mu,\nu) \leq c \left( C(\alpha q, s) \sqrt{n} \right)^{\frac{2|s|\alpha'(q-p)}{1+2|s|\alpha'(q-p)}} \left(\max \left\{\sqrt{n}, \frac{1}{|s|} \right\} \right)^{\frac{q-p}{q (1+2|s|\alpha'(q-p))}} \, W_p(\mu,\nu)^{\frac{p}{q} \frac{1}{1 + 2|s|\alpha'(q-p)}}, $$
where \[ C(\alpha q,s)=\left\{
\begin{array}{cl}
 \alpha q & \mbox{ for } \, s > - \frac{1}{\alpha q+1} \\
\frac{1}{|s|^{1-\frac{1}{\alpha q}}(1- \alpha q |s|)^{\frac{1}{\alpha q}}} & \mbox{ for } \, -\frac{1}{\alpha q} < s \leq - \frac{1}{\alpha q+1} \\
\end{array}. \right.\]
Here $\alpha' = \frac{\alpha}{\alpha-1}$ denotes the H\"older conjugate of $\alpha$.

\end{theorem} 

\begin{proof}
We follow \cite{MM} with the necessary modifications. Let $X$ and $Y$ be distributed according to $\mu$ and $\nu$ respectively. Note that for all $R>0$,
$$ \E[|X-Y|^q] = \E[|X-Y|^q 1_{\{|X-Y| \leq R\}}] + \E[|X-Y|^q 1_{\{|X-Y| > R\}}]. $$
On one hand,
\begin{equation}\label{wass-1}
\E[|X-Y|^q 1_{\{|X-Y| \leq R\}}] \leq R^{q-p} \E[|X-Y|^p].
\end{equation}
On the other hand, by H\"older's inequality,
\begin{equation}\label{wass-1-2}
\E[|X-Y|^q 1_{\{|X-Y| > R\}}] \leq \E[|X-Y|^{\alpha q}]^{\frac{1}{\alpha}} P(|X-Y|>R)^{\frac{1}{\alpha'}}.
\end{equation}
Since $P(|X-Y|>R) \leq P(|X| > R/2) + P(|Y| > R/2)$, Lemma \ref{deviation} implies that there is a universal constant $c>0$ such that
\begin{equation}\label{wass-2}P(|X-Y|>R) \leq 2 \left( \frac{2c \max(\sqrt{n}, \frac{1}{|s|})}{R} \right)^{\frac{1}{2|s|}}.
\end{equation}
By isotropicity of $X$ and Lemma \ref{norm}, there is a universal constant $c>0$ such that
\begin{equation}\label{wass-3} \E[|X-Y|^{\alpha q}]^{\frac{1}{\alpha}} \leq \left( \E[|X|^{\alpha q}]^{\frac{1}{\alpha q}} + \E[|Y|^{\alpha q}]^{\frac{1}{\alpha q}} \right)^q \leq \left( c \, C(\alpha q,s) \sqrt{n} \right)^{q}.
\end{equation}
Combining \eqref{wass-1}, \eqref{wass-1-2}, \eqref{wass-2} and \eqref{wass-3}, we obtain
$$ \E[|X-Y|^q] \leq R^{q-p} \E[|X-Y|^p] + 2^{\frac{1}{\alpha'}} \left( \frac{2c \max(\sqrt{n},\frac{1}{|s|})}{R} \right)^{\frac{1}{2|s|\alpha'}} \, \left( c \, C(\alpha q,s) \sqrt{n} \right)^{q}. $$
Taking infimimum over all coupling results in
$$ W_q(\mu,\nu)^q \leq \inf_{R>0} \left[ R^{q-p} W_p(\mu,\nu)^p + 2^{\frac{1}{\alpha'}} \left( \frac{2c \max(\sqrt{n},\frac{1}{|s|})}{R} \right)^{\frac{1}{2|s|\alpha'}} \, \left( c \, C(\alpha q,s) \sqrt{n} \right)^{q} 
\right]. $$
By Lemma \ref{minimize}, we deduce that
$$ W_q(\mu,\nu)^q \leq W_p(\mu,\nu)^{\frac{p}{1+ 2|s|\alpha'(q-p)}} \left( 2^{\frac{1}{\alpha'}}  (2c\max\{\sqrt{n}, \frac{1}{|s|}\})^{\frac{1}{2|s|\alpha'}} [c \, C(\alpha q,s)\sqrt{n} ]^q \right)^{\frac{q-p}{\frac{1}{2|s|\alpha'} + q-p}} $$
$$ \qquad \qquad \times \left[ \left( \frac{1}{2|s|\alpha' (q-p)} \right)^{\frac{q-p}{\frac{1}{2|s|\alpha'} + q-p}}+ \left( 2|s|\alpha'(q-p) \right)^{\frac{1}{1 + 2|s|\alpha'(q-p)}} \right]. $$
We conclude by using the fact that $m^{-\frac{m}{1+m}} + m^{\frac{1}{1+m}} \leq 2$, with $m = 2|s|\alpha'(q-p) > 0$.
\end{proof}

\begin{remark}

Similarly as in Remark \ref{rem-1}, as $s \to 0$, the constant in the right-hand side of Theorem \ref{compar-wass} blows up to $+\infty$, and one needs to apply the second part of Lemma \ref{deviation} to recover the result for log-concave measures. We do not know whether Theorem \ref{compar-wass} is valid when $s$ is in a neighborhood of $-\frac{1}{q}$.

\end{remark}

We now discuss entropic distances. In general, one cannot compare the relative entropy $D(\mu || \nu)$ and, say, $d_{TV}(\mu , \nu)$, for arbitrary $s$-concave measures $\mu, \nu$, since $d_{TV}(\mu , \nu) \leq 2$ while $D(\mu || \nu) = +\infty$ if $\mu$ is not absolutely continuous with respect to $\nu$. Next, we establish quantitative comparisons for relative entropy and more general R\'enyi divergences when $\nu = \gamma_n$ the standard Gaussian measure in $\R^n$. The quantity $D(\mu || \gamma_n)$ is of fundamental importance as it is strongly related to the hyperplane conjecture (see, e.g., \cite{BM11}, \cite{MK}, \cite{BM20}) and to the entropic Central Limit Theorem (see, e.g., \cite{Barr}, \cite{ABB} \cite{EMZ}). The following result provides a comparison between the relative entropy and total variation distance.

\begin{theorem}\label{compar-rel-ent}

Let $\alpha \in (1,2]$. Let $s \in (-\frac{1}{2\alpha}, 0)$. Let $\mu$ be an isotropic $s$-concave probability measure in $\R^{n}$, and let $\gamma_{n}$ denote the standard Gaussian distribution in $\R^{n}$. Then, there is a universal constant $c>0$ such that
$$ D(\mu || \gamma_{n}) \leq \frac{c n (1+n|s|) \log(\alpha' n)}{(1 - 2\alpha |s|)^{\frac{4|s|(\alpha'-1)}{1+4|s|\alpha'}}} \max \left( 1, \frac{1}{\sqrt{n} |s|} \right)^{\frac{2}{1+4|s|\alpha'}} \left( d_{TV}(\mu, \gamma_n)^{\frac{1}{1+4|s|\alpha'}} + d_{TV}(\mu, \gamma_n) \right), $$
where $\alpha' = \frac{\alpha}{\alpha-1}$ denotes the H\"older conjugate of $\alpha$.

\end{theorem}

\begin{proof}
We follow \cite{MM} but correct a mistake in their original argument. Let us denote by $f$ the density of $\mu$ and denote
$$ \phi(x) = (2\pi)^{-\frac{n}{2}} e^{-\frac{|x|^{2}}{2}}, \quad x \in \R^n, $$
the density of the standard Gaussian measure $\gamma_{n}$ in $\R^n$. Denote by $Z$ a random variable with density $\phi$, by $Y$ a random variable with density $f$, and denote
$$ X = \frac{f(Z)}{\phi(Z)}, \quad W = \frac{f(Y)}{\phi(Y)}. $$
Note that
\begin{equation}\label{D} \E[X \log(X)] = \E \left[ \frac{f(Z)}{\phi(Z)} \log \left( \frac{f(Z)}{\phi(Z)} \right) \right] = \int f(x) \log \left( \frac{f(x)}{\phi(x)} \right) dx = D(\mu || \gamma_{n}). \end{equation}
It is classical that if $\mu$ and $\nu$ have densities $u$ and $v$ respectively (with respect to Lebesgue measure), then $d_{TV}(\mu , \nu) = \int |u-v| dx$, therefore
\begin{equation}\label{d-TV}
d_{TV}(\mu , \gamma_{n}) = \int |f(x) - \phi(x)| dx = \E[|X-1|] = 2 \E[(X-1) 1_{\{X \geq 1\}}],
\end{equation}
where we use $\E[X] = 1$ in the last equality. Now, consider the function $h(x)= x \log(x)$ on $[0, +\infty)$. Since $h$ is convex and $h(1) = 0$, we have for all $1 \leq x \leq R$,
$$ h(x) \leq \frac{h(R)}{R-1}(x-1). $$
Hence, if $R \geq 2$, then for all $1 \leq x \leq R$, $x \log(x) \leq 2\log(R)(x-1)$. Therefore, using \eqref{D} and \eqref{d-TV},
\begin{eqnarray}\label{D-estim}
D(\mu || \gamma_{n}) = \E[X \log(X)] & \leq & \E[X \log(X) 1_{\{1 \leq X \leq R\}}] + \E[X \log(X) 1_{\{X > R\}}] \nonumber \\
& \leq & 2\log(R) \E[(X-1) 1_{\{1 \leq X \leq R\}}] + \E[X \log(X) 1_{\{X > R\}}] \nonumber \\
& \leq & \log(R) \, d_{TV}(\mu,\gamma_{n}) + \E[X \log(X) 1_{\{X > R\}}].
\end{eqnarray}
Since $W = f(Y)/\phi(Y)$, we have
\begin{eqnarray}\label{D-est}
\E[X \log(X) 1_{\{X > R\}}] & = & \int f(x) \log \left( \frac{f(x)}{\phi(x)} \right) 1_{\{f > R \phi\}}(x) dx \nonumber \\ & = & \E[\log(W) 1_{\{W > R\}}] \nonumber \\ & \leq & \E[|\log(W)|^{\alpha}]^{\frac{1}{\alpha}} \P(W > R)^{1-\frac{1}{\alpha}},
\end{eqnarray}
where the last inequality follows from H\"older's inequality. Next, we are going to upper bound the term $\E[|\log(W)|^{\alpha}]^{1/\alpha}$. Note that
$$ \E[|\log(\phi(Y))|^{\alpha}]^{\frac{1}{\alpha}} \leq \frac{n}{2} \log(2 \pi) + \frac{1}{2} \E[|Y|^{2\alpha}]^{\frac{1}{\alpha}}. $$
Since $Y$ is isotropic, we deduce by Lemma \ref{norm} that there exists a universal constant $c>0$ such that
\begin{equation}\label{first}
\E[|\log(\phi(Y))|^{\alpha}]^{\frac{1}{\alpha}} \leq \frac{cn}{(1 - 2\alpha |s|)^{\frac{1}{\alpha}}}.
\end{equation}
On the other hand, by H\"older's inequality and Lemma \ref{variance},
\begin{equation}\label{log2}
\E[|\log(f(Y))|^{\alpha}]^{\frac{2}{\alpha}} \leq \E[\log^2(f(Y))] = \Var(\log(f(Y))) + \E[\log(f(Y))]^2 \leq \frac{n}{(1+n \kappa)^2} + h(Y)^2,
\end{equation}
where $\kappa = \frac{s}{1 - sn}$, and $h(Y) = \E[-\log(f(Y))]$ is the differential entropy of $Y$.

Here, we correct a mistake from \cite{MM}, where it is claimed that the inequality $h(Y) \geq 0$ holds for isotropic log-concave $Y$, which is being used to obtain an upper bound on $h(Y)^2$. This is inaccurate, and in fact, the inequality $h(Y) \geq 0$ would imply the hyperplane conjecture (see, e.g., \cite{BM11}, \cite{MK}, \cite{BM20}).

Since Gaussians maximize the entropy when fixing the covariance matrix, we have
\begin{equation}\label{gauss}
    h(Y) \leq \frac{n}{2}\log(2 \pi e).
\end{equation} 
On the other hand, by Proposition \ref{max-dens}, there is a universal constant $c > 0$ such that
\begin{equation}\label{max-max}
\|f\|_{\infty} \leq c^{n(1+n|s|)} d_0^n n^{n/2},
\end{equation}
where $d_0$ is defined in \eqref{c_0-}, hence
\begin{equation}\label{down}
h(Y) = \E[-\log(f(Y))] \geq = -\frac{n}{2} \log(c^{2(1+n|s|)} d_0^2 \, n).
\end{equation}
Since $s>-\frac{1}{2\alpha}$, we have
\begin{equation}\label{bound-d0}
    d_0 \leq \alpha' (1+n|s|)^{4(1+n|s|)},
\end{equation}
therefore, combining \eqref{gauss} with \eqref{down}, we deduce the existence of an absolute $c>0$ such that
$$ h(Y)^2 \leq c \, n^2 (1+n|s|)^2 \log(\alpha' n)^2, $$
where $\alpha'$ is the H\"older conjugate of $\alpha$. In particular, there should be an extra $\log(n)$ factor in the proof of \cite[Proposition 7]{MM}. Recalling \eqref{log2}, we deduce that
\begin{equation}\label{second}
    \E[|\log(f(Y))|^{\alpha}]^{\frac{1}{\alpha}} \leq c \, n (1+n|s|) \log(\alpha' n),
\end{equation}
for some absolute constant $c>0$. Therefore, combining \eqref{first} and \eqref{second},
\begin{eqnarray}\label{1st}
\E[|\log(W)|^{\alpha}]^{\frac{1}{\alpha}} \leq \frac{c}{(1 - 2\alpha |s|)^{\frac{1}{\alpha}}} \, n (1+n|s|) \log(\alpha' n),
\end{eqnarray}
for some absolute constant $c>0$. It remains to upper bound $P(W > R)$. By Lemma \ref{deviation} and \eqref{max-max}, there is a universal constant $c>0$ such that for all $R \geq c^{n(1+n|s|)} d_0^n (2\pi n)^{\frac{n}{2}}$,
\begin{eqnarray}\label{2nd}
\P(W > R) & \leq & \P \left( (2\pi)^{\frac{n}{2}} e^{\frac{|Y|^2}{2}} > \frac{R}{c^{n(1+n|s|)} d_0^n n^{\frac{n}{2}}} \right) \nonumber \\ & = & \P \left( |Y| > \sqrt{2 \log \left( \frac{R}{c^{n(1+n|s|)} d_0^n (2\pi n)^{\frac{n}{2}}} \right)} \right) \nonumber \\ & \leq & \left[ \frac{ c \max\{\sqrt{n}, \frac{1}{|s|} \}}{\sqrt{2 \log \left( \frac{R}{c^{n(1+n|s|)} d_0^n (2\pi n)^{\frac{n}{2}}} \right)}} \right]^{\frac{1}{2|s|}}.
\end{eqnarray}
Finally, combining \eqref{D-estim}, \eqref{D-est}, \eqref{1st} and \eqref{2nd}, there is an absolute constant $c>0$ such that for all $R \geq c^{n(1+n|s|)} d_0^n (2\pi n)^{\frac{n}{2}}$,
\begin{eqnarray*}
D(\mu || \gamma_n) & \leq & \log(R) d_{TV}(\mu, \gamma_n) + \E[|\log(W)|^{\alpha}]^{\frac{1}{\alpha}} \P(W > R)^{\frac{1}{\alpha'}} \\ & \leq & \log(R) d_{TV}(\mu, \gamma_n) + \frac{c n (1+n|s|) \log(\alpha' n)}{(1 - 2\alpha |s|)^{\frac{1}{\alpha}}} \frac{ \left(c \max\{\sqrt{n}, \frac{1}{|s|} \} \right)^{\frac{1}{2|s|\alpha'}}}{\log \left( \frac{R}{c^{n(1+n|s|)} d_0^n (2\pi n)^{\frac{n}{2}}} \right)^{\frac{1}{4|s|\alpha'}}} \\ & = & A t + AM + \frac{B}{t^p},
\end{eqnarray*}
where
$$ A = d_{TV}(\mu, \gamma_n), t = \log \left( \frac{R}{c^{n(1+n|s|)} d_0^n (2\pi n)^{\frac{n}{2}}} \right), M = \log(c^{n(1+n|s|)} d_0^n (2\pi n)^{\frac{n}{2}}), $$
$$ B = \frac{c n (1+n|s|) \log(\alpha' n)}{(1-2\alpha |s|)^{\frac{1}{\alpha}}} \left(c \max\{\sqrt{n}, \frac{1}{|s|} \} \right)^{\frac{1}{2|s| \alpha'}}, p=\frac{1}{4|s| \alpha'}. $$
Minimizing over $R > c^{n(1+n|s|)} d_0^n (2\pi n)^{\frac{n}{2}}$, or, equivalently, over $t>0$, we have by Lemma \ref{minimize},
\begin{eqnarray*} D(\mu || \gamma_n) \leq A^{\frac{p}{p+1}} B^{\frac{1}{p+1}} \left( p^{\frac{1}{p+1}} + \left( \frac{1}{p} \right)^{\frac{p}{1+p}} \right) + AM,
\end{eqnarray*}
and the result follows using \eqref{bound-d0}.
\end{proof}

We do not know whether Theorem \ref{compar-rel-ent} holds when $s$ is in a neighborhood of $-\frac{1}{2}$ with a rate of convergence independent of $\alpha$. Nonetheless, under an exponential moment assumption and a weaker rate of convergence, one may provide a comparison (dependent on the exponential moment) between more general R\'enyi divergences and total variation distance for $s \in (-\frac{1}{2}, 0)$. This is the aim of the next theorem.

\begin{theorem}\label{compar-Tsal}

Let $s \in (-\frac{1}{2},0)$. Let $\mu$ be an isotropic $s$-concave probability measure on $\R^{n}$, and let $\gamma_{n}$ denote the standard Gaussian distribution on $\R^{n}$. Let $p > 1$ and $\alpha \in (1,2]$. Under the moment assumption
$$ M = \int e^{\frac{|x|^2}{2} \alpha(p-1)} d\mu(x) < +\infty, $$
we have, denoting $d_{TV} = d_{TV}(\mu, \gamma_n)$,
$$ T_{p}(\mu || \gamma_{n}) \leq \frac{\left( c^{n(1+n|s|)} d_0^n n^{\frac{n}{2}} \right)^{p-1}}{p-1} \left[ \sqrt{d_{TV} + d_{TV}^2} + M^{\frac{1}{\alpha}} \left( \frac{c \sqrt{p-1} \max\{\sqrt{n}, \frac{1}{|s|} \}}{\sqrt{\log \left(1 + \frac{1}{d_{TV}} \right)}} \right)^{\frac{1}{2 |s| \alpha'}} \right], $$
where $c>0$ is an absolute constant and $d_0$ is defined in \eqref{c_0-}.

\end{theorem}

\begin{proof}
Recall the definition of the Tsallis entropy of order $p > 1$ of $Y$ with density $f$ and $Z$ with density $\phi$,
$$ T_{p}(Y || Z) = \frac{1}{p - 1} \left[ \int \frac{f(x)^{p}}{\phi(x)^{p -1}} dx - 1 \right]. $$
Denote, as in the proof of Theorem \ref{compar-rel-ent},
$$ X = \frac{f(Z)}{\phi(Z)}, \quad W = \frac{f(Y)}{\phi(Y)}. $$
We have, for $R \geq 1$,
$$ (p - 1) T_{p}(Y || Z) = \E[X^{p} - 1] \leq \E[(X^{p} - 1)1_{\{1 \leq X \leq R\}}] + \E[(X^{p} - 1)1_{\{X > R\}}]. $$
Note that the function $h(x) = x^{p} - 1$ is convex, and $h(1)=0$, therefore, we have that for all $1 \leq x \leq R$,
$$ h(x) \leq \frac{h(R)}{R-1}(x-1). $$
Hence, for all $R \geq 2$, recalling \eqref{d-TV},
$$ \E[(X^{p} - 1)1_{\{1 \leq X \leq R\}}] \leq 2 \frac{R^{p} - 1}{R} \E[(X-1)1_{\{1 \leq X \leq R\}}] \leq R^{p-1} d_{TV}(\mu, \gamma_n). $$
On the other hand,
$$ \E[X^{p} 1_{\{X > R\}}] = \E[W^{p - 1} 1_{\{W > R\}}] \leq \E[W^{\alpha(p - 1)}]^{\frac{1}{\alpha}} \P(W > R)^{\frac{\alpha - 1}{\alpha}}. $$
Using Proposition \ref{max-dens}, we have for some absolute constant $c>0$, 
$$ \E[W^{\alpha (p - 1)}]^{\frac{1}{\alpha}} = \E \left[ \left( \frac{f(Y)}{\phi(Y)} \right)^{\alpha (p - 1)} \right]^{\frac{1}{\alpha}} \leq \left( c^{n(1+n|s|)} d_0^n n^{\frac{n}{2}} \right)^{p-1} \E[e^{\frac{|Y|^2}{2} \alpha (p - 1)}]^{\frac{1}{\alpha}}, $$
and by \eqref{2nd}, we have for all $R > c^{n(1+n|s|)} d_0^n (2\pi n)^{\frac{n}{2}}$,
$$ \P(W > R)  \leq \left[ \frac{ c \max\{\sqrt{n}, \frac{1}{|s|} \}}{\sqrt{2 \log \left( \frac{R}{c^{n(1+n|s|)} d_0^n (2\pi n)^{\frac{n}{2}}} \right)}} \right]^{\frac{1}{2|s|}}. $$
Hence,
$$ (p - 1) T_{p}(Y || Z) \leq R^{p-1} \, d_{TV}(\mu, \gamma_n) + \frac{A}{\log \left( \frac{R}{c^{n(1+n|s|)} d_0^n (2\pi n)^{\frac{n}{2}}} \right)^{\frac{1}{4 |s| \alpha'}}}, $$
where
$$ A = \left( c^{n(1+n|s|)} d_0^n n^{\frac{n}{2}} \right)^{p-1} \E[e^{\frac{|Y|^2}{2} \alpha (p - 1)}]^{\frac{1}{\alpha}} \left( c \max\{\sqrt{n}, \frac{1}{|s|} \} \right)^{\frac{1}{2|s|\alpha'}}. $$
It remains to choose
$$ R = c^{n(1+n|s|)} d_0^n (2\pi n)^{\frac{n}{2}} \left( 1 + \frac{1}{d_{TV}(\mu, \gamma_n)} \right)^{\frac{1}{2 (p-1)}}. $$
\end{proof}

Let us comment on the moment assumption in Theorem \ref{compar-Tsal}. It was shown in \cite{BCG} that the assumption $T_{p}(Y||Z) < +\infty$, $p>1$, implies that $Y$ has moments of all orders, and in fact
$$ \E[e^{c\frac{|Y|^2}{2}}] < +\infty, $$
for all $c < \frac{p-1}{p}$. On the other hand, note that if $Y$ has a bounded density, which is the case if $Y$ is $s$-concave, then the assumption
$$ \E[e^{\frac{|Y|^2}{2}(p-1)}] < +\infty, $$
implies that $T_{p}(X||Z) < +\infty$.

\vskip1.5cm

\noindent Arnaud Marsiglietti \\
Department of Mathematics \\
University of Florida \\
Gainesville, FL 32611, USA \\
a.marsiglietti@ufl.edu

\vspace{0.8cm}

\noindent Puja Pandey \\
Department of Mathematics \\
University of Florida \\
Gainesville, FL 32611, USA \\
puja.pandey@ufl.edu

\end{document}